\documentclass[11pt]{article}
\usepackage{amsthm,amstext}
\usepackage{amsmath}
\usepackage{graphicx}
\usepackage{amsfonts}
\usepackage{amssymb}
\theoremstyle{plain}
\newtheorem{theorem}{Theorem}[section]

\newtheorem{proposition}[theorem]{Proposition}

\theoremstyle{definition}
\newtheorem{definition}[theorem]{Definition}

\newtheorem{example}{\sc Example}
\theoremstyle{remark}
\newtheorem{remark}{\sc Remark}

\date{}
\title{\bf On $Q$-Fuzzy Ideals in $\Gamma$-Semigroups}\vspace{.25 in}
\author{ {\bf Samit Kumar Majumder}\\
Tarangapur N.K High School, Tarangapur,\\
Uttar Dinajpur, West Bengal-733 129, INDIA\\
{\tt samitfuzzy@gmail.com}
 }

\begin{document}
\maketitle

\begin{abstract}

In this paper the concept of $Q$-fuzzification of ideals of $\Gamma$-semigroups has been introduced and some important properties have been investigated. A characterization of regular $\Gamma$-semigroup in terms of $Q$-fuzzy ideals has been obtained. Operator semigroups of a $\Gamma$-semigroup has been made to work by obtaining various relationships between $Q$-fuzzy ideals of a $\Gamma$-semigroup and that of its operator semigroups. \\

\textbf{AMS Mathematics Subject Classification(2000):}\textit{\ }08A72,
20M12, 3F55\\

\textbf{Key Words and Phrases:}\textit{\ } $Q$-Fuzzy set, $\Gamma$-Semigroup, Regular $\Gamma$-semigroup, $Q$-Fuzzy left$($right$)$ ideal, Operator semigroups.
\end{abstract}

\section{Introduction}
A  semigroup is an algebraic structure consisting of a non-empty
set $S$ together with an associative binary operation\cite{H}. The
formal study of semigroups began in the early 20th century.
Semigroups are important in many areas of mathematics, for
example, coding and language theory, automata theory,
combinatorics and mathematical analysis. The concept of fuzzy sets
was introduced by \textit{Lofti Zadeh}\cite{Z} in his classic
paper in 1965. \textit{Azirel Rosenfeld}\cite{R} used the idea of
fuzzy set to introduce the notions of fuzzy subgroups.
\textit{Nobuaki Kuroki}\cite{K1,K2,K3} is the pioneer of fuzzy
ideal theory of semigroups. The idea of fuzzy subsemigroup was
also introduced by \textit{Kuroki}\cite{K1,K3,Mo}. In \cite{K2},
\textit{Kuroki} characterized several classes of semigroups in
terms of fuzzy left, fuzzy right and fuzzy bi-ideals. Others who
worked on fuzzy semigroup theory, such as \textit{X.Y.
Xie}\cite{X1,X2}, \textit{Y.B. Jun}\cite{J}, are mentioned in the
bibliography. \textit{X.Y. Xie}\cite{X1} introduced the idea of
extensions of fuzzy ideals in semigroups.

The notion of a $\Gamma$-semigroup was introduced by \textit{Sen} and
\textit{Saha}\cite{Sen2} as a generalization of semigroups and ternary
semigroup. $\Gamma$-semigroup have been analyzed by lot of
mathematicians, for instance by \textit{Chattopadhyay}\cite{C1,C2}, \textit{Dutta} and
\textit{Adhikari}\cite{D1,D2}, \textit{Hila}\cite{H1,H2}, \textit{Chinram}\cite{Ch}, \textit{Saha}\cite{Sa}, \textit{Sen} et al.\cite{Sen1,Sen2,Sa}, \textit{Seth}\cite{Se}. \textit{S.K. Sardar} and \textit{S.K. Majumder}\cite{D3,D4,S1,S2} have introduced the notion of fuzzification of ideals, prime ideals, semiprime ideals and ideal extensions of $\Gamma$-semigroups and studied them via its operator semigroups. In this paper the concept of $Q$-fuzzy ideals of a $\Gamma$-semigroup has been introduced. It is observed here that they satisfy level subset crieterion as well as characteristic function crieterion. Finally in order to make operator semigroups of a $\Gamma$-semigroup to work in the context of $Q$-fuzzy sets as it worked in the study of $\Gamma$-semigroups\cite{D1,D2}, we obtain various relationships between $Q$-fuzzy ideals of a $\Gamma$-semigroup and that of its operator semigroups. Here, among other results we obtain an inclusion preserving bijection between the set of all $Q$-fuzzy ideals of a $\Gamma$-semigroup and that of its operator semigroups.
%%%%%%%%%%%%%%%%%%%%%%%%%%%%%%%%%%%%%%%%%%%%%%%%%%%%%%%%%%%%%%%%%%%%%%%%%%%%%%%%%%%%%%%%%%
\section{Preliminaries}

In this section we discuss some elementary definitions that we use
in the sequel.\\

\begin{definition}
\cite{Sen1} Let $S=\{x,y,z,.....\}$ and
$\Gamma=\{\alpha,\beta,\gamma,.....\}$ be two non-empty sets. Then
$S$ is called a $\Gamma$-semigroup if there exist a mapping
$S\times\Gamma\times S\rightarrow S($images are denoted by
$a\alpha b)$ satisfying $(1)$ $x\gamma y\in S,$ $(2)$ $(x\beta
y)\gamma z=x\beta(y\gamma z)$ for all $x,y,z\in S$ and
$\beta,\gamma\in\Gamma.$
\end{definition}

\begin{example}
Let $\Gamma=\{5,7\}.$ For any $x,y\in N$ and $\gamma\in\Gamma,$ we define $x\gamma y=x.\gamma.y$ where $.$ is the usual multiplication on $N.$ Then $N$ is a $\Gamma$-semigroup.
\end{example}

\begin{remark}
Definition $2.1$ is the definition of one sided
$\Gamma$-semigroup. Both sided $\Gamma$-semigroup was defined by
Dutta and Adhikari\cite{D1} where the operation $\Gamma\times
S\times\Gamma\rightarrow\Gamma$ also taken into consideration.
They defined operator semigroups for such $\Gamma$-semigroups. The
following definition is the definition of both sided
$\Gamma$-semigroup given by Dutta and Adhikari.
\end{remark}

\begin{definition}
\cite{D1} Let $S$ and $\Gamma$ be two non-empty sets. $S$ is called
a $\Gamma$-semigroup if there exist mappings from
$S\times\Gamma\times S$ to $S,$ written as
$(a,\alpha,b)\rightarrow a\alpha b,$ and from $\Gamma\times
S\times\Gamma\rightarrow\Gamma,$ written as
$(\alpha,a,\beta)\rightarrow\alpha a\beta$ satisfying the
following associative laws: $(a\alpha b)\beta c=a(\alpha b)\beta
c=a\alpha (b\beta c)$ and $\alpha(a\beta b)\gamma=(\alpha a\beta)
b\gamma=\alpha a(\beta b\gamma)$ for all $a,b,c\in S$ and for all
$\alpha,\beta,\gamma\in\Gamma.$
\end{definition}

\begin{example}
\cite{D1} Let $S$ be the set of all integers of the form $4n+1$ and $\Gamma$ be the set of all integers of the form $4n+3$ where $n$ is an integer. If $a\alpha b$ is $a+\alpha+b$ and $\alpha+a+\beta($usual sum of integers$)$ for all $a,b\in S$ and for all $\alpha,\beta\in\Gamma.$ Then $S$ is a $\Gamma$-semigroup.
\end{example}

\begin{definition}
\cite{D1} Let $S$ be a $\Gamma$-semigroup. By a left$($right$)$ ideal of $S$ we mean a non-empty subset $A$ of $S$ such that $S\Gamma A\subseteq A(A\Gamma S\subseteq A).$ By a two sided ideal or simply an ideal, we mean a non-empty subset of $S$ which is both a left and right ideal of $S.$
\end{definition}

\begin{definition}
\cite{D1} A $\Gamma$-semigroup is called regular if, for each element $x\in S,$ there exist $\beta\in\Gamma$ such that $x=x\beta x.$
\end{definition}

\begin{definition}
\cite{Z} A fuzzy subset $\rho$ of a non-empty set $X$ is a function $\rho:X\rightarrow [0,1].$
\end{definition}

\begin{definition}
Let $\rho$ be a fuzzy subset of a non-empty set $X.$ Then the set $\rho_{t}=\{x\in X:\rho(x)\geq t\}$ for $t\in [0,1],$ is called the level subset or $t$-level subset of $\rho.$
\end{definition}

\begin{definition}
Let $Q$ and $X$ be two non-empty sets. A mapping $\mu:X\times Q\rightarrow [0,1]$ is called the $Q$-fuzzy subset of $X.$
\end{definition}

\begin{definition}
Let $\mu$ be a $Q$-fuzzy subset of a non-empty set $X.$ Then the set $\mu_{t}=\{x\in X:\mu(x,q)\geq t\forall q\in Q\}$ for $t\in [0,1],$ is called the level subset or $t$-level subset of $\mu.$
\end{definition}

\begin{example}
Let $S=\{a,b,c\}$ and $\Gamma=\{\gamma,\delta\},$ where $\gamma,\delta$ is defined on $S$ with the following caley table: \\\\
$%
\begin{tabular}{|c|c|c|c|}
\hline
$\gamma $ & $a$ & $b$ & $c$ \\ \hline
$a$ & $a$ & $a$ & $a$ \\ \hline
$b$ & $b$ & $b$ & $b$ \\ \hline
$c$ & $c$ & $c$ & $c$ \\ \hline
\end{tabular}%
$ \ \ \ \ and \ \ \ \  $%
\begin{tabular}{|c|c|c|c|}
\hline
$\delta $ & $a$ & $b$ & $c$ \\ \hline
$a$ & $a$ & $a$ & $a$ \\ \hline
$b$ & $b$ & $b$ & $b$ \\ \hline
$c$ & $c$ & $c$ & $c$ \\ \hline
\end{tabular}%
$\\

Then $S$ is a $\Gamma$-semigroup. Let $Q=\{p\}.$ Let us consider a $Q$-fuzzy subset $\mu: S\times Q\rightarrow [0,1],$ by $\mu(a,p)=0.8,\mu(b,p)=0.7,\mu(c,p)=0.6.$ For $t=0.7,$ $\mu_{t}=\{a,b\}.$
\end{example}

%%%%%%%%%%%%%%%%%%%%%%%%%%%%%%%%%%%%%%%%%%%%%%%%%%%%%%%%%%%%%%%%%%%%%%%%%%%%%%%%%%%%%%

\section{$Q$-Fuzzy Ideals}

\begin{definition}
A non-empty $Q$-fuzzy subset $\mu$ of a $\Gamma$-semigroup $S$ is called a $Q$-fuzzy left ideal of $S$ if $\mu(x\gamma y,q)\geq\mu(y,q)\forall x,y\in S,\forall\gamma\in\Gamma$ and $\forall q\in Q.$
\end{definition}

\begin{definition}
A non-empty $Q$-fuzzy subset $\mu$ of a $\Gamma$-semigroup $S$ is called a $Q$-fuzzy right ideal of $S$ if $\mu(x\gamma y,q)\geq\mu(x,q)\forall x,y\in S,\forall\gamma\in\Gamma$ and $\forall q\in Q.$
\end{definition}

\begin{definition}
A non-empty $Q$-fuzzy subset of a $\Gamma$-semigroup $S$ is called a $Q$-fuzzy ideal of $S$ if it is both a $Q$-fuzzy left ideal and a $Q$-fuzzy right ideal of $S.$
\end{definition}

\begin{example}
Let $S$ be the set of all non-positive integers and $\Gamma$ be the set of all non-positive even integers. Then $S$ is a $\Gamma$-semigroup if $a\gamma b$ and $\alpha a\beta$ denote the usual multiplication of integers $a,\gamma,b$ and $\alpha,a,\beta$ respectively where $a,b\in S$ and $\alpha,\beta,\gamma\in\Gamma.$ Let $Q=\{p\}.$ Let $\mu$ be a $Q$-fuzzy subset of $S$ defined as follows:

\begin{align*}
\mu(x,p)=\left\{
\begin{array}{ll}
1 & \text{if} \ x=0 \\
0.1 & \text{if} \ x=-1,-2 \\
0.2 &  \text{if} \ x<-2
\end{array}
\right..
\end{align*}

Then $\mu$ is a $Q$-fuzzy ideal of $S.$
\end{example}

\begin{theorem}
Let $I$ be a non-empty subset of a $\Gamma$-semigroup $S$ and $\chi_{I\times Q}$ be the characteristic function of $I\times Q$ then $I$ is a left ideal$($right ideal, ideal$)$ of $S$ if and only if $\chi_{I\times Q}$ is a $Q$-fuzzy left ideal$($resp. $Q$-fuzzy right ideal, fuzzy ideal$)$ of $S$.
\end{theorem}

\begin{proof}
Let $I$ be a left ideal of a $\Gamma$-semigroup $S.$ Let $x,y\in S,q\in Q$ and $\gamma\in\Gamma,$ then $x\gamma y\in I$ if $y\in I.$ It follows that $\chi_{I\times Q}(x\gamma y,q)=1=\chi_{I\times Q}(y,q).$ If $y\notin I,$ then $\chi_{I\times Q}(y,q)=0.$ In this case $\chi_{I\times Q}(x\gamma y,q)\geq0=\chi_{I\times Q}(y,q)$ $.$ Therefore $\chi_{I\times Q}$ is a $Q$-fuzzy left ideal of $S.$

Conversely, let $\chi_{I\times Q}$ be a $Q$-fuzzy left ideal of $S.$ Let $x,y\in I,q\in Q,$ then
$\chi_{I\times Q}(x,q)=\chi_{I\times Q}(y,q)=1.$ Now let $x\in I$ and $s\in S,\gamma\in\Gamma,q\in Q.$
Then $\chi_{I\times Q}(x,q)=1.$ Also $\chi_{I\times Q}(s\gamma x,q)\geq\chi_{I\times Q}(x,q)=1.$ Thus
$s\gamma x\in I.$ So $I$ is a left ideal of $S.$ Similarly we can prove that the other parts of the theorem.
\end{proof}

\begin{proposition}
Let $I$ be a left ideal$($right ideal, ideal$)$ of a $\Gamma$-semigroup $S,$ $Q$ be any non-empty set
and $\alpha\leq\beta\neq0$ be any two elements in $[0,1]$, then the $Q$-fuzzy
subset $\mu$ of $\ S,$ defined by $\mu(x,q)$ =$\left\{
\begin{array}
[l]{l}%
\beta,\text{if }x\in I,q\in Q\\
\alpha,\text{otherwise}%
\end{array}
\right.  $ is a $Q$-fuzzy left ideal$($resp. $Q$-fuzzy right ideal, $Q$-fuzzy ideal$)$ of
$S.$
\end{proposition}

\begin{proof}
Let $I$ be a left ideal of a $\Gamma$-semigroup $S$ and $\alpha,\beta\in\lbrack0,1].$ Let $s,x\in S,q\in Q$ and $\gamma\in\Gamma.$ If $x\in I,$ then $s\gamma x\in I$ and $\mu(x,q)=\beta=\mu(s\gamma x,q).$ Therefore $\mu(s\gamma
x,q)=\mu(x,q).$ If $x\notin I$ then $\mu(x,q)=\alpha\leq\beta$ and then $\mu(s\gamma x,q)\geq\mu(x,q).$ Thus $\mu(s\gamma x,q)\geq\mu(x,q)$ for all $x,s\in S,q\in Q$ and for all $\gamma\in\Gamma.$ Hence $\mu$ is a $Q$-fuzzy left ideal of $S.$ Similarly we can prove all other cases.
\end{proof}

\begin{theorem}
Let $S$ be a $\Gamma$-semigroup, $Q$ be any non-empty set and $\mu$ be a non-empty $Q$-fuzzy subset of $S,$ then $\mu$ is a $Q$-fuzzy left ideal$(Q$-fuzzy right ideal, $Q$-fuzzy ideal) of $S$ if and only if $\mu_{t}$'s are left ideals$($resp. right ideals, ideals$)$ of $S$ for all $t\in Im(\mu),$ where $\mu_{t}=\{x\in S:\mu(x,q)\geq t\forall q\in Q\}.$
\end{theorem}

\begin{proof}
Let $\mu$ be a $Q$-fuzzy left ideal of $S.$ Let $t\in\operatorname{Im}\mu,$ then there exist some $\alpha\in S$ such that $\mu(\alpha,q)=t$ and so $\alpha\in\mu_{t.}$ Thus $\mu_{t}\neq\phi.$ Let $x,y\in\mu_{t}$, then $\mu(x,q)\geq t$ and $\mu(y,q)\geq t.$ Again let $s\in S,x\in\mu_{t}$ and $\gamma\in\Gamma.$ Now
$\mu(s\gamma x,q)\geq\mu(x,q)\geq t.$ Therefore $s\gamma x\in\mu_{t.}$ Thus
$\mu_{t}$ is a left ideal of $S.$

Conversely, let $\mu_{t}$'s are left ideals of $S$ for all $t\in\operatorname{Im}\mu.$ Again let $x\in S,s\in S$ and $\gamma\in\Gamma,$ then $\mu(x,q)=t\forall q\in Q.$ Thus $s\gamma x\in\mu_{t}$ $($since $\mu_{t}$ is a left ideal of
$S).$ Therefore $\mu(s\gamma x,q)\geq t=\mu(x,q).$ Hence $\mu$ is a $Q$-fuzzy left ideal of $S.$ Similarly we can prove the other cases.
\end{proof}

\begin{remark}
Theorem  and  are true in case of semigroup also.
\end{remark}

%%%%%%%%%%%%%%%%%%%%%%%%%%%%%%%%%%%%%%%%%%%%%%%%%%%%%%%%%%%%%%%%%%%%%%%%%%%%%%%%%%%

\section{Composition of $Q$-Fuzzy Ideals}

In this section we define composition of $Q$-fuzzy ideals of a $\Gamma$-semigroup and characterize regular $\Gamma$-semigroups in terms of $Q$-fuzzy ideals.\newline

\begin{definition}
Let $S$ be a $\Gamma$-semigroup, $\mu_{1},\mu_{2}\in QFLI(S)[QFRI(S),QFI(S)]$ and $Q$ be a non-empty set
\footnote{$QFLI(S)$, $QFRI(S)$, $QFI(S)$ denote respectively the set of all $Q$-fuzzy left ideals, $Q$-fuzzy right ideals, $Q$-fuzzy ideals of a $\Gamma$-semigroup $S.$}. Then the product $\mu_{1}\circ\mu_{2}$ of $\mu_{1}$ and $\mu_{2}$ is defined as
\end{definition}

$(\mu_{1}\circ\mu_{2})(x,q)=\left\{
\begin{array}
[c]{c}%
\underset{x=u\gamma v}{\sup}[\min\{\mu_{1}(u,q),\mu_{2}(v,q)\}:u,v\in S;\gamma
\in\Gamma;q\in Q]\\
0,\text{ if for any }u,v\in S\text{ and for any }\gamma\in\Gamma,x\neq u\gamma
v
\end{array}
\right.  $

\begin{theorem}
Let $S$ be a $\Gamma$-semigroup and $Q$ be any non-empty set. Then following are equivalent: $(1)$ $\mu$ is a $Q$-fuzzy left$($right$)$ ideal of $S,$ $(2)$ $\chi\circ\mu\subseteq\mu(\mu\circ\chi\subseteq\mu),$ where $\chi$ is the characteristic function of $S\times Q.$
\end{theorem}

\begin{proof}
Let $\mu$ be a $Q$-fuzzy left ideal of $S$. Let $a\in S,q\in Q.$ Suppose there exist
$u,v\in S$ and $\delta\in\Gamma$ such that $a=u\delta v.$ Then, since $\mu$ is
a $Q$-fuzzy left ideal of $S,$ we have%
\begin{align*}
(\chi\circ\mu)(a,q)  &  =\underset{a=x\gamma y}{\sup}[\min\{\chi(x,q),\mu(y,q)\}]\\
&  =\underset{a=x\gamma y}{\sup}[\min\{1,\mu(y,q)\}]=\underset{a=x\gamma y}%
{\sup}{\mu(y,q)}.
\end{align*}
Now, since $\mu$ is a $Q$-fuzzy left ideal, $\mu(x\gamma y,q)\geq\mu(y,q)$ for all
$x,y\in S,q\in Q$ and for all $\gamma\in\Gamma.$ So in particular, $\mu(y,q)\leq\mu(a,q)$
for all $a=x\gamma y$. Hence $\underset{a=x\gamma y}{\sup}{\mu(y,q)}\leq\mu(a,q).$
Thus $\mu(a,q)\geq(\chi\circ\mu)(a,q)$. If there do not exist $x,y\in S,\gamma
\in\Gamma$ such that $a=x\gamma y$ then $(\chi\circ\mu)(a,q)=0\leq\mu(a,q).$ Hence
$\chi\circ\mu\subseteq\mu.$ By a similar argument we can show that $\mu
\circ\chi\subseteq\mu$ when $\mu$ is a $Q$-fuzzy right ideal of $S.$

Conversely, let $\chi\circ\mu\subseteq\mu.$ Let $x,y\in S,q\in Q,$ $\gamma\in\Gamma$
and $a:=x\gamma y.$ Then $\mu(x\gamma y,q)=\mu(a,q)\geq(\chi\circ\mu)(a,q).$ Now%
\begin{align*}
(\chi\circ\mu)(a,q)  &  =\underset{a=u\alpha v}{\sup}[\min\{\chi(u,q),\mu
(v,q)\}]\geq\min\{\chi(x,q),\mu(y,q)\}\\
&  =\min\{1,\mu(y,q)\}=\mu(y,q).
\end{align*}
Hence $\mu(x\gamma y,q)\geq\mu(y,q)$. Hence $\mu$ is a $Q$-fuzzy left ideal of $S.$ By
a similar argument we can show that if $\mu\circ\chi\subseteq\mu,$ then $\mu$
is a $Q$-fuzzy right ideal of $S$.\newline
\end{proof}

Using the above theorem we can deduce the following theorem.

\begin{theorem}
Let $S$ be a $\Gamma$-semigroup and $Q$ be any non-empty set. Then following are equivalent: $(1)$ $\mu$ is a
$Q$-fuzzy two-sided ideal of $S,$ \ \ $(2)$ $\chi\circ\mu\subseteq\mu$ and
$\mu\circ\chi\subseteq\mu,$ where $\chi$ is the characteristic function of
$S\times Q.$
\end{theorem}

\begin{proposition}
Let $Q$ be any non-empty set, $\mu_{1}$ be a $Q$-fuzzy right ideal and $\mu_{2}$ be a $Q$-fuzzy left ideal of a
$\Gamma$-semigroup $S$. Then $\mu_{1}\circ\mu_{2}\subseteq\mu_{1}\cap\mu_{2}.$
\end{proposition}

\begin{proof}
Let $\mu_{1}$ be a $Q$-fuzzy right ideal and $\mu_{2}$ be a $Q$-fuzzy left ideal of a
$\Gamma$-semigroup $S.$ Let $x\in S$ and $q\in Q.$ Suppose there exist $u_{1},v_{1}\in S$
and $\gamma_{1}\in\Gamma$ such that $x=u_{1}\gamma_{1}v_{1}$. Then%
\begin{align*}
(\mu_{1}\circ\mu_{2})(x,q)  &  =\underset{x=u\gamma v}{\sup}\min\{\mu_{1}%
(u,q),\mu_{2}(v,q)\}\\
&  \leq\underset{x=u\gamma v}{\sup}\min\{\mu_{1}(u\gamma v,q),\mu_{2}(u\gamma
v,q)\}\\
&  =\min\{\mu_{1}(x,q),\mu_{2}(x,q)\}=(\mu_{1}\cap\mu_{2})(x,q).
\end{align*}
Suppose there do not exist $u,v\in S$ such that $x=u\gamma v$. Then $(\mu
_{1}\circ\mu_{2})(x,q)=0\leq(\mu_{1}\cap\mu_{2})(x,q).$ Thus $\mu_{1}\circ\mu
_{2}\subseteq\mu_{1}\cap\mu_{2}.$\newline
\end{proof}

From the above proposition and the definition of $\mu_{1}\cap\mu_{2}$ the
following proposition follows easily.

\begin{proposition}
Let $Q$ be a non-empty set and $\mu_{1},\mu_{2}\in QFI(S)$. Then $\mu_{1}\circ\mu_{2}\subseteq\mu_{1}%
\cap\mu_{2}\subseteq\mu_{1},\mu_{2}.$
\end{proposition}

\begin{proposition}
Let $S$ be a regular $\Gamma$-semigroup, $Q$ be any non-empty set and $\mu_{1},\mu_{2}$ be two $Q$-fuzzy
subsets of $S$. Then $\mu_{1}\circ\mu_{2}\supseteq\mu_{1}\cap\mu_{2}$.
\end{proposition}

\begin{proof}
Let $c\in S$ and $q\in Q.$ Since $S$ is regular, then there exists an element $x\in S$ and
$\gamma_{1},\gamma_{2}\in\Gamma$ such that $c=c\gamma_{1}x\gamma_{2}c=c\gamma
c$ where $\gamma:=\gamma_{1}x\gamma_{2}\in\Gamma$. Then%
\begin{align*}
(\mu_{1}\circ\mu_{2})(c,q)  &  =\underset{c=u\alpha v}{\sup}\{\min\{\mu
_{1}(u,q),\mu_{2}(v,q)\}\}\\
&  \geq\min\{\mu_{1}(c,q),\mu_{2}(c,q)\}=(\mu_{1}\cap\mu_{2})(c,q).
\end{align*}
Therefore $\mu_{1}\circ\mu_{2}\supseteq\mu_{1}\cap\mu_{2}$.
\end{proof}

\begin{theorem}
Let $S$ be a $\Gamma$-semigroup and $Q$ be any non-empty set. Then following are equivalent. $(1)$ $S$ is
regular. $(2)$ $\mu_{1}\circ\mu_{2}=\mu_{1}\cap\mu_{2}$ for every $Q$-fuzzy right
ideal $\mu_{1}$ and every $Q$-fuzzy left ideal $\mu_{2}$ of $S.$
\end{theorem}

\begin{proof}
Let $S$ be a regular $\Gamma$-semigroup. Then by Proposition $4.6$, $\mu
_{1}\circ\mu_{2}\supseteq\mu_{1}\cap\mu_{2}$. Again by Proposition $4.4$,
$\mu_{1}\circ\mu_{2}\subseteq\mu_{1}\cap\mu_{2}.$ Hence $\mu_{1}\circ\mu
_{2}=\mu_{1}\cap\mu_{2}.$

Conversely, let $S$ be a $\Gamma$-semigroup, $q\in Q$ and for every $Q$-fuzzy right ideal
$\mu_{1}$ and every $Q$-fuzzy left ideal $\mu_{2}$ of $S,$ $\mu_{1}\circ\mu
_{2}=\mu_{1}\cap\mu_{2}.$ Let $L$ and $R$ be respectively a left ideal and a
right ideal of $S$ and $x\in R\cap L.$ Then $x\in R$ and $x\in L.$ Hence
$\chi_{L\times Q}(x,q)=\chi_{R\times Q}(x,q)=1($where $\chi_{L\times Q}(x)$ and $\chi_{R\times Q}(x)$ are
respectively the characteristic functions of $L\times Q$ and $R\times Q)$. Thus%
\[
(\chi_{R\times Q}\cap\chi_{L\times Q})(x)=\min\{\chi_{R\times Q}(x),\chi_{L\times Q}(x)\}=1.
\]
Now by Theorem $3.4$, $\chi_{L\times Q}$ and $\chi_{R\times Q}$ are respectively a $Q$-fuzzy
left ideal and a $Q$-fuzzy right ideal of $S$. Hence by hypothesis, $\chi_{R\times Q}%
\circ\chi_{L\times Q}=\chi_{R\times Q}\cap\chi_{L\times Q}$. Hence%
\begin{align*}
(\chi_{R\times Q}\circ\chi_{L\times Q})(x,q)  &  =1\\
i.e.,\underset{x=y\gamma z}{\sup}[\min\{\chi_{R\times Q}(y,q),\chi_{L\times Q}(z,q)\}  &  :y,z\in
S;\gamma\in\Gamma]=1.
\end{align*}

This implies that there exist some $r,s\in S,q\in Q$ and $\gamma_{1}\in\Gamma$ such
that $x=r\gamma_{1}s$ and $\chi_{R\times Q}(r,q)=1=\chi_{L\times Q}(s,q)$. Hence $r\in R$ and
$s\in L.$ Hence $x\in R\Gamma L.$ Thus $R\cap L\subseteq R\Gamma L.$ Also
$R\Gamma L\subseteq R\cap L.$ Hence $R\Gamma L=R\cap L.$ Consequently, the
$\Gamma$-semigroup $S$ is regular.
\end{proof}

%%%%%%%%%%%%%%%%%%%%%%%%%%%%%%%%%%%%%%%%%%%%%%%%%%%%%%%%%%%%%%%%%%%%%%%%%%%%%%%%%%%%%%

\section{Corresponding $Q$-Fuzzy Ideals}

Many results of semigroups could be extended to $\Gamma$-semigroups directly
and via operator semigroups\cite{D1} of a $\Gamma$-semigroup. In
this section in order to make operator semigroups of a $\Gamma$-semigroup
work in the context of $Q$-fuzzy sets as it worked in the study of $\Gamma
$-semigroups$\cite{D1,D2}$, we obtain various
relationships between $Q$-fuzzy ideals of a $\Gamma$-semigroup and that of its
operator semigroups. Here, among other results we obtain an inclusion
preserving bijection between the set of all $Q$-fuzzy ideals of a $\Gamma
$-semigroup and that of its operator semigroups.

\begin{definition}
$\cite{D1}$Let $S$ be a $\Gamma$-semigroup. Let us define a relation
$\rho$ on $S\times\Gamma$ as follows : $(x,\alpha)\rho(y,\beta)$ if and only
if $x\alpha s=y\beta s$ for all $s\in S$ and $\gamma x\alpha=\gamma y\beta$
for all $\gamma\in\Gamma.$ Then $\rho$ is an equivalence relation. Let
$[x,\alpha]$ denote the equivalence class containing $(x,\alpha)$. Let
$L=\{[x,\alpha]:x\in S,\alpha\in\Gamma\}.$ Then $L$ is a semigroup with
respect to the multiplication defined by $[x,\alpha][y,\beta]=[x\alpha
y,\beta].$\textit{ }This semigroup $L$ is called the left operator semigroup
of the $\Gamma$-semigroup $S.$ Dually the right operator semigroup $R$ of
$\Gamma$-semigroup $S$ is defined where the multiplication is defined by
$[\alpha,a][\beta,b]=[\alpha a\beta,b].$

If there exists an element $[e,\delta]\in L([\gamma,f]\in R)$ such that $e\delta s=s($resp. $s\gamma f=s)$ for all $s\in S$ then $[e,\delta]($resp. $[\gamma,f])$ is called the left$($resp. right$)$ unity\index{unity(left, right)} of $S.$
\end{definition}

\begin{definition}
Let $Q$ be any non-empty set. For a $Q$-fuzzy subset $\mu$ of $R$ we define a $Q$-fuzzy subset $\mu^{\ast}$ of $S$
by $\mu^{\ast}(a,q)=\underset{\gamma\in\Gamma}{\inf}$ $\mu([\gamma,a],q),$ where
$a\in S,q\in Q.$ For a $Q$-fuzzy subset $\sigma$ of $S$ we define a $Q$-fuzzy subset
$\sigma^{\ast^{^{\prime}}}$ of $R$ by $\sigma^{\ast^{^{\prime}}}%
([\alpha,a],q)=\underset{s\in S}{\inf}$ $\sigma(s\alpha a,q),$ where
$[\alpha,a]\in R,q\in Q.$ For a $Q$-fuzzy subset $\delta$ of $L,$ we define a $Q$-fuzzy
subset $\delta^{+}$ of $S$ by $\delta^{+}(a,q)=\underset{\gamma\in\Gamma}{\inf}$
$\delta([a,\gamma],q)$ where $a\in S,q\in Q$. For a $Q$-fuzzy subset $\eta$\textit{ }of $S$
we define a $Q$-fuzzy subset $\eta^{+^{^{\prime}}}$ of $L$ by $\eta^{+^{^{\prime}%
}}([a,\alpha],q)=\underset{s\in S}{\inf}$ $\eta(a\alpha s,q),$ where
$[a,\alpha]\in L,q\in Q.$
\end{definition}

Now we recall the following propositions from $\cite{D1}$ which were
proved therein for one sided ideals. But the results can be proved to be true
for two sided ideals.\newline

\begin{proposition}
$\cite{D1}$Let $S$ be a $\Gamma$-semigroup with unities and $L$ be
its left operator semigroup. If $A$ is a $($right$)$ ideal of $L$ then $A^{+}$
is a $($right$)$ideal of $S.$
\end{proposition}

\begin{proposition}
$\cite{D1}$Let $S$ be a $\Gamma$-semigroup with unities and $L$ be
its left operator semigroup. If $B$ is a $($right$)$ideal of $S$ then
$B^{+^{^{\prime}}}$ is a$($right$)$ideal of $L.$
\end{proposition}

\begin{proposition}
$\cite{D1}$Let $S$ be a $\Gamma$-semigroup with unities and $R$ be
its right operator semigroup. If $A$ is a $($left$)$ideal of $R$ then
$A^{\ast}$ is a $($left$)$ideal of $S.$
\end{proposition}

\begin{proposition}
$\cite{D1}$Let $S$ be a $\Gamma$-semigroup with unities and $R$ be
its right operator semigroup. If $B$ is a $($left$)$ideal of $S$ then
$B^{\ast^{^{\prime}}}$ is a $($left$)$ideal of $R.$
\end{proposition}

For convenience of the readers, we may note that for a $\Gamma$-semigroup $S$
and its left, right operator semigroups $L,R$ respectively four mappings
namely $()^{+},$ $()^{+^{^{\prime}}},()^{+},()^{\ast^{^{\prime}}}$ occur. They
are defined as follows: For $I\subseteq R,I^{\ast}=\{s\in S,[\alpha,s]\in
I\forall\alpha\in\Gamma\}$; for $P\subseteq S,P^{\ast^{^{\prime}}}%
=\{[\alpha,x]\in R:s\alpha x\in P\forall s\in S\}$; for $J\subseteq
L,J^{+}=\{s\in S,[s,\alpha]\in J\forall\alpha\in\Gamma\}$; for $Q\subseteq
S,Q^{+^{^{\prime}}}=\{[x,\alpha]\in L:x\alpha s\in Q\forall s\in S\}.$

\begin{proposition}
Let $Q$ be any non-empty set and $\mu$ is a $Q$-fuzzy subset of $R($the right operator semigroup of the $\Gamma
$-semigroup $S).$ Then $(\mu_{t})^{\ast}=(\mu^{\ast})_{t}$ for all
$t\in\lbrack0,1]$ such that the sets are non-empty.
\end{proposition}

\begin{proof}
Let $s\in S,q\in Q.$ Then%
\begin{align*}
s  &  \in(\mu_{t})^{\ast}\Leftrightarrow\lbrack\gamma,s]\in\mu_{t}\text{
}\forall\gamma\in\Gamma\Leftrightarrow\mu([\gamma,s],q)\geq t\text{ }%
\forall\gamma\in\Gamma\\
&  \Leftrightarrow\underset{\gamma\in\Gamma}{\inf}\mu([\gamma,s],q)\geq
t\Leftrightarrow\mu^{\ast}(s,q)\geq t\Leftrightarrow s\in(\mu^{\ast})_{t}.
\end{align*}
Hence $(\mu_{t})^{\ast}=(\mu^{\ast})_{t}.$\newline
\end{proof}

\begin{proposition}
Let $Q$ be any non-empty set and $\sigma$ is a $Q$-fuzzy subset of a $\Gamma$-semigroup $S$. Then $(\sigma
_{t})^{\ast^{^{\prime}}}=(\sigma^{\ast^{^{\prime}}})_{t}$ for all $t\in
\lbrack0,1]$ such that the sets under consideration are non-empty.
\end{proposition}

\begin{proof}
Let $[\alpha,x]\in R,q\in Q$ and $t$ is as mentioned in the statement. Then%
\begin{align*}
\lbrack\alpha,x]  &  \in(\sigma_{t})^{\ast^{^{\prime}}}\Leftrightarrow s\alpha
x\subseteq\sigma_{t}\text{ }\forall s\in S\\
&  \Leftrightarrow\sigma(s\alpha x,q)\geq t\text{ }\forall s\in S\Leftrightarrow
\underset{s\in S}{\inf}\sigma(s\alpha x,q)\geq t\\
&  \Leftrightarrow\sigma^{\ast^{^{\prime}}}([\alpha,x],q)\geq t\Leftrightarrow
\lbrack\alpha,x]\in(\sigma^{\ast^{^{\prime}}})_{t}.
\end{align*}
Hence $(\sigma_{t})^{\ast^{^{\prime}}}=(\sigma^{\ast^{^{\prime}}})_{t}$
$.\newline$
\end{proof}

In what follows $S$ denotes a $\Gamma$-semigroup with unities$\cite{D1},L\,,R$ be its left and right operator semigroups respectively.

\begin{proposition}
If $Q$ be any non-empty set and $\mu\in QFI(R)(QFLI(R))$, then $\mu^{\ast}\in QFI(S)($respectively
$QFLI(S))$.
\end{proposition}

\begin{proof}
Suppose $\mu\in QFI(R)$. Then $\mu_{t}$ is an ideal of $R,$ $\ \forall t\in
{Im}(\mu).$ Hence $(\mu_{t})^{\ast}$ is an ideal of $S,$ $\forall t\in{Im}%
(\mu)(cf.$ Proposition $5.5).$ Let $q\in Q.$ Now since $\mu$ is $Q$-fuzzy ideal of $R$, $\mu$
is a non-empty $Q$-fuzzy subset of $R$. Hence for some $[\alpha,s]\in
R,\mu([\alpha,s],q)>0$. Then $\mu_{t}\neq\phi$ where $t:=\mu([\alpha,s],q)$. So by
the same argument applied above $(\mu_{t})^{\ast}\neq\phi$. Let $u\in(\mu
_{t})^{\ast}$. Then $[\beta,u]\in\mu_{t}$ for all $\beta\in\Gamma$. Hence
$\mu([\beta,u],q)\geq t$. This implies that $\underset{\beta\in\Gamma}{\inf}%
\mu([\beta,u],q)\geq t,$ $i.e.,$ $\mu^{\ast}(u,q)\geq t$. Hence $u\in(\mu^{\ast
})_{t}.$ Hence$(\mu^{\ast})_{t}\neq\phi$. Consequently, $(\mu_{t})^{\ast}%
=(\mu^{\ast})_{t}(cf.$ Proposition $5.7).$ It follows that $(\mu^{\ast}%
)_{t}$ is an ideal of $S$ for all $t\in{Im}(\mu).$ Hence $\mu^{\ast}$ is a
$Q$-fuzzy ideal of $S(cf.$ Theorem $3.6)$. The proof for $Q$-fuzzy left ideal
follows similarly.\newline
\end{proof}

\ In a similar fashion we can deduce the following proposition.

\begin{proposition}
If $Q$ be any non-empty set and $\sigma\in QFI(S)(QFLI(S)),$ then $\sigma^{\ast^{^{\prime}}}\in QFI(R)
($respectively $QFLI(R))$.
\end{proposition}

We can also deduce the following left operator analogues of the above propositions.

\begin{proposition}
If $Q$ be any non-empty set and $\delta\in QFI(L)(QFRI(L))$, then $\delta^{+}\in QFI(S)($respectively $QFRI(S))$.
\end{proposition}

\begin{proposition}
If $Q$ be any non-empty set and $\eta\in QFI(S)(QFRI(S))$, then $\eta^{+^{^{\prime}}}\in QFI(L)
($respectively $QFRI(L)).$
\end{proposition}

\begin{theorem}
Let $S$ be a $\Gamma$-semigroup with unities, $Q$ be any non-empty set and $L$ be its left operator
semigroup. Then there exist an inclusion preserving bijection $\sigma
\mapsto\sigma^{+^{^{\prime}}}$ between the set of all $Q$-fuzzy ideals $(Q$-fuzzy
right ideals$)$ of $S$ and set of all $Q$-fuzzy ideals $($resp. $Q$-fuzzy right
ideals$)$ of $L,$ where $\sigma$ is a $Q$-fuzzy ideal $($resp. $Q$-fuzzy right
ideal$)$ of $S.$
\end{theorem}

\begin{proof}
Let $\sigma\in QFI(S)(QFRI(S))$ and $x\in S,q\in Q.$ Then%
\[
(\sigma^{+^{^{\prime}}})^{+}(x,q)=\underset{\gamma\in\Gamma}{\inf}%
\sigma^{+^{^{\prime}}}([x,\gamma],q)=\underset{\gamma\in\Gamma}{\inf}%
\underset{s\in S}{[\inf}\sigma(x\gamma s,q)]\geq\sigma(x,q).
\]
Hence $\sigma\subseteq(\sigma^{+^{^{\prime}}})^{+}.$ Let $[\gamma,f]$ be the
right unity of $S.$ Then $x\gamma f=x$ for all $x\in S.$ Then%
\[
\sigma(x,q)=\sigma(x\gamma f,q)\geq\underset{\alpha\in\Gamma}{\inf}\underset{s\in
S}{[\inf}\sigma(x\alpha s,q)]=\underset{\alpha\in\Gamma}{\inf}\sigma
^{+^{^{\prime}}}([x,\alpha],q)=(\sigma^{+^{^{\prime}}})^{+}(x,q).
\]
So $\sigma\supseteq(\sigma^{+^{^{\prime}}})^{+}.$ Hence $(\sigma^{+^{^{\prime
}}})^{+}=\sigma.$ Now let $\mu\in QFI(L)(QFRI(L)).$ Then%
\begin{align*}
(\mu^{+})^{+^{^{\prime}}}([x,\alpha],q)  &  =\underset{s\in S}{\inf}\mu
^{+}(x\alpha s,q)=\underset{s\in S}{\inf}[\underset{\gamma\in\Gamma}{\inf}%
\mu([x\alpha s,\gamma],q)\\
&  =\underset{s\in S}{\inf}[\underset{\gamma\in\Gamma}{\inf}\mu([x,\alpha
][s,\gamma],q)]\geq\mu([x,\alpha],q).
\end{align*}
So $\mu\subseteq(\mu^{+})^{+^{^{\prime}}}.$ Let $[e,\delta]$ be the left unity
of $L.$ Then%
\begin{align*}
\mu([x,\alpha],q)  &  =\mu([x,\alpha][e,\delta],q)\geq\underset{s\in S}{\inf
}[\underset{\gamma\in\Gamma}{\inf}\mu([x,\alpha][s,\gamma],q)]\\
&  =(\mu^{+})^{+^{^{\prime}}}([x,\alpha],q).
\end{align*}
So $\mu\supseteq(\mu^{+})^{+^{^{\prime}}}$ and hence $\mu=(\mu^{+}%
)^{+^{^{\prime}}}.$ Thus the correspondence $\sigma\mapsto\sigma^{+^{^{\prime
}}}$ is a bijection. Now let $\sigma_{1},\sigma_{2}\in QFI(S)(QFRI(S))$ be such
that $\sigma_{1}\subseteq\sigma_{2}.$ Then for all $[x,\alpha]\in L,q\in Q$%
\[
\sigma_{1}^{+^{^{\prime}}}([x,\alpha],q)=\underset{s\in S}{\inf}\sigma
_{1}(x\alpha s,q)\leq\underset{s\in S}{\inf}\sigma_{2}(x\alpha s,q)=\sigma
_{2}^{+^{^{\prime}}}([x,\alpha],q).
\]
Thus $\sigma_{1}^{+^{^{\prime}}}\subseteq\sigma_{2}^{+^{^{\prime}}}.$
Similarly we can show that if $\mu_{1}\subseteq\mu_{2}$ where $\mu_{1},\mu
_{2}\in QFI(L)(QFRI(L))$ then $\mu_{1}^{+}\subseteq\mu_{2}^{+}.$ Hence
$\sigma\mapsto\sigma^{+^{^{\prime}}}$ is an inclusion preserving bijection.
The rest of the proof follows from Proposition $5.11$ and Proposition
$5.12$.\newline
\end{proof}

In a similar way by using Proposition $5.9$ and Proposition $5.10$ we can
deduce the following theorem.

\begin{theorem}
Let $S$ be a $\Gamma$-semigroup with unities, $Q$ be any non-empty set and $R$ be its right operator
semigroup. Then there exist an inclusion preserving bijection $\sigma
\mapsto\sigma^{\ast^{^{\prime}}}$ between the set of all $Q$-fuzzy ideals $(Q$-fuzzy
left ideals$)$ of $S$ and set of all $Q$-fuzzy ideals $($resp. $Q$-fuzzy left
ideals$)$ of $R,$ where $\sigma$ is a $Q$-fuzzy ideal $($resp. $Q$-fuzzy left ideal$)$
of $S.$
\end{theorem}

%%%%%%%%%%%%%%%%%%%%%%%%%%%%%%%%%%%%%%%%%%%%%%%%%%%%%%%%%%%%%%%%%%%%%%%%%%%%%%%%%%
%%%%%%%%%%%%%%%%%%%%%%%%%%%%%%%%%%%%%%%%%%%%%%%%%%%%%%%%%%%%%%%%%%%%%%%%%%%%%%%%%%%%%%%%%

\end{document}